\documentclass[11pt,a4paper]{amsart}

\usepackage{geometry}
\usepackage{mathtools,graphicx}
\mathtoolsset{showonlyrefs}	
\numberwithin{equation}{section}

\usepackage[normalem]{ulem}

\theoremstyle{plain}
\newtheorem{dfn}{Definition}[section]
\newtheorem{prop}[dfn]{Proposition}
\newtheorem{thm}[dfn]{Theorem}
\newtheorem{lem}[dfn]{Lemma}
\newtheorem{cor}[dfn]{Corollary}

\theoremstyle{definition}
\newtheorem{rmk}[dfn]{Remark}

\newenvironment{prf}{\begin{proof}[{\it Proof: \nopunct}]}{\end{proof}}

\def\bC          {\mathbb C}
\def\bP          {\mathbb P}
\def\bZ          {\mathbb Z}

\def\cA           {\mathcal A}
\def\cF           {\mathcal F}
\def\cI           {\mathcal I}
\def\cJ           {\mathcal J}

\def\fH{\mathfrak H}
\def\fp         {\mathfrak p}

\DeclareMathOperator{\PSL}{PSL}
\DeclareMathOperator{\SL}{SL}

\title{An algebraic characterization of the Kronecker function}

\author{Nils Matthes}
\address{Department of Mathematics, Kyushu University, 744 Motooka, Nishi-Ku, Fukuoka 819-0395, Japan}
\address{
Current address: Mathematical Institute, University of Oxford,
Andrew Wiles Building, Radcliffe Observatory Quarter, Woodstock Road, Oxford OX2 6GG, United Kingdom}
\email{nils.matthes@maths.ox.ac.uk}

\begin{document}
\begin{abstract}
We characterize Zagier's generating series of extended period polynomials of normalized Hecke eigenforms for $\PSL_2(\bZ)$ in terms of the period relations and existence of a suitable factorization. For this we use a characterization of the Kronecker function as the ``fundamental solution'' of the Fay identity.
\end{abstract}
\maketitle

\section{Introduction}

\subsection{Zagier's generating series of extended period polynomials}

To every cusp form\footnote{Here and in what follows every modular form will be for the full modular group $\PSL_2(\bZ):=\SL_2(\bZ)/\{\pm 1\}$.} $f$ one can attach its period polynomial $r_f(X)$. This is an important arithmetic invariant whose coefficients are, up to elementary factors, the critical values of the completed L-function of $f$. Period polynomials were studied extensively by Eichler, Shimura and Manin (see for example \cite[Chapter V]{Lang:ModularForms} for a textbook account) and their definition was extended to all modular forms by Zagier, \cite{Zagier:Periods}. Nowadays, extended period polynomials are well-studied objects in number theory which keep attracting attention. For example, they are related to zeta functions of real quadratic fields, \cite{KZ}, and multiple zeta values, \cite{GKZ}. More recently, an analogue of the Riemann hypothesis was proved for them, \cite{CFI,Ono}.

In \cite{Zagier:Periods}, Zagier introduced a generating series $C(X,Y;\tau;T)$ of extended period polynomials of normalized Hecke eigenforms (see Section \ref{ssec:4.1} for its definition), and related it to the Kronecker function, \cite[Chapter VIII]{Weil:Elliptic},
\begin{equation}
F_{\tau}(u,v):=\frac{\theta'_{\tau}(0)\theta_{\tau}(u+v)}{\theta_{\tau}(u)\theta_{\tau}(v)}.
\end{equation}
Here $\theta_{\tau}(u)$ denotes Jacobi's odd theta function (see Section \ref{ssec:2.1} for the definition). More precisely, Zagier's main result is as follows.
\begin{thm}[Zagier; {\cite{Zagier:Periods}, eqn. (17)}] \label{thm:Zagier}
We have the following identity of formal series in variables $X,Y,T$:
\begin{equation} \label{eqn:ZagiersTheorem}
C(X,Y;\tau;T)=F_{\tau}(T,-XYT)F_{\tau}(XT,YT).
\end{equation}
\end{thm}
As two applications, we mention the algorithmic computation (up to scalars) of the period polynomials of all cuspidal Hecke eigenforms, as well as a new proof of the Eichler--Selberg trace formula for $\PSL_2(\bZ)$, \cite{Zagier:Hecke}.
\subsection{Period relations and the main result}
The goal of this paper is to give another application of Zagier's result, namely a purely algebraic characterization of $C$ in terms of the period relations for extended period polynomials
\begin{equation} \label{eqn:periodrels}
\begin{aligned}
r_f(X)\vert(1+S)=r_f(X)\vert(1+U+U^2)=0,
\end{aligned}
\end{equation}
where $S=\left(\begin{smallmatrix}0&-1\\1&0\end{smallmatrix}\right)$, $U=\left(\begin{smallmatrix}1&-1\\1&0\end{smallmatrix}\right)$, and $\vert$ denotes the slash operator. Indeed, since $C$ is defined as a generating series of extended period polynomials it also satisfies a suitable version of the period relations, see Equation \eqref{eqn:periodrelations}.

It turns out that the period relations alone do not suffice to characterize $C$, as for any modular form $f$ of weight $k$ the series $C+r_f(X)T^{k-2}$ also satisfies the period relations. However, if we demand in addition the existence of a factorization as in Theorem \ref{thm:Zagier}, then the situation is much better and our main result is as follows.
\begin{thm} \label{thm:main2}
For a formal Laurent series $f(u,v)=\sum_{m,n>\!\!>-\infty}a_{m,n}u^mv^n \in \bC((u,v))$, assume that
\begin{equation}
C_f(X,Y,T):=f(T,-XYT)f(XT,YT) \in \bC((X,Y,T))
\end{equation}
satisfies the period relations, Equation \eqref{eqn:periodrelations}. Then either:
\begin{enumerate}
\item[(i)]
There exist unique $\alpha,\beta \in \bC$ such that
\begin{equation} \label{eqn:case1}
C_f(X,Y,T)=C_{P_{\alpha,\beta}}(X,Y,T)=\frac{(\alpha XY-\beta)(\beta X+\alpha Y)}{X^2Y^2T^2},
\end{equation}
where $P_{\alpha,\beta}(u,v)=\alpha u^{-1}+\beta v^{-1}$.
\item[(ii)]
There exist $\alpha,\beta,\delta \in \bC^{\times}$ and $\tau \in \fH \cup \{i\infty\}$ such that
\begin{equation} \label{eqn:case2}
C_f(X,Y,T)=C_{F_{\tau}}(X,Y',T')=C(X,Y';\tau;T'), \quad \mbox{where } T':=\frac{T}{\alpha}, \, Y':=\frac{\alpha Y}{\beta},
\end{equation}
Moreover, if $\alpha',\beta',\delta',\tau'$ are a different choice of parameters as above, then there exists a matrix $\left(\begin{smallmatrix}a&b\\c&d\end{smallmatrix}\right) \in \SL_2(\bZ)$ such that
\begin{equation}
(\alpha',\beta',\delta',\tau')=\left( \alpha(c\tau+d),\beta(c\tau+d),\frac{\delta}{c\tau+d},\frac{a\tau+b}{c\tau+d} \right).
\end{equation}
\end{enumerate}
\end{thm}
The basic idea is that, via Theorem \ref{thm:Zagier} above, the period relations for $C$ are equivalent to the Fay identity for $F_{\tau}(u,v)$, Equation \eqref{eqn:Fay}, which is a consequence of Fay's trisecant identity for theta functions, \cite[p.34, eqn. (45)]{Fay:Theta}. More precisely,  $C_f$ satisfies the period relations if and only if $f$ satisfies the Fay identity, see Proposition \ref{prop:periodsFay}.
We are therefore led to characterize solutions of the Fay identity, for which we have the following result.
\begin{thm} \label{thm:main}
Assume that $f(u,v)=\sum_{m,n>\!\!>-\infty}a_{m,n}u^mv^n$ satisfies the Fay identity, Equation \eqref{eqn:Fay}.
\begin{enumerate}
\item[(i)]
If $(a_{3,0},a_{5,0})=(0,0)$, then there exist unique $\alpha,\beta,\gamma \in \bC$ such that
\begin{equation} \label{eqn:thmmainpolar}
f(u,v)=e^{\gamma uv}\left(\frac{\alpha}{u}+\frac{\beta}{v}\right).
\end{equation}
\item[(ii)]
If $(a_{3,0},a_{5,0}) \neq (0,0)$, then there exist $\alpha,\beta,\delta \in \bC^{\times}$, $\gamma \in \bC$ and $\tau \in \fH \cup \{i\infty\}$ such that
\begin{equation} \label{eqn:thmmainelliptic}
f(u,v)=\delta e^{\gamma uv}F_{\tau}\left(\frac u\alpha, \frac v\beta\right).
\end{equation}
Moreover, if $\alpha',\beta',\gamma',\delta',\tau'$ are a different choice of parameters as above, then there exists $\left(\begin{smallmatrix}a&b\\c&d\end{smallmatrix}\right) \in \SL_2(\bZ)$ such that
\begin{equation}
(\alpha',\beta',\gamma',\delta',\tau')=\left( \alpha(c\tau+d),\beta(c\tau+d),\gamma-\frac{c/2\pi i}{\alpha\beta(c\tau+d)},\frac{\delta}{c\tau+d},\frac{a\tau+b}{c\tau+d} \right).
\end{equation}
\end{enumerate}
\end{thm}
The idea of proof is that the Fay identity implies recurrence relations for the coefficients of $f$ such that every solution is uniquely determined by at most five of its coefficients (see Proposition \ref{prop:Faybasic2} and Theorem \ref{thm:generators} for the precise statements). Varying the parameters $\alpha,\beta,\gamma,\delta,\tau$ suitably, we can arrange that $f$ is equal to one of the functions in Theorem \ref{thm:main}.

After a previous version of this manuscript had been submitted, the author was informed that the existence part of Theorem \ref{thm:main} is equivalent to a result of Polishchuk, \cite[Theorem 5]{Polishchuk:YB}. To be precise, Polishchuk works with the (scalar) associative Yang--Baxter equation which is equivalent to the Fay identity, \cite{Polishchuk:Tangents}, and classifies solutions in the space of meromorphic functions defined in a neighborhood of $(0,0)$ instead of formal Laurent series.
The two proofs are quite similar but ours may still be of independent interest as it yields slightly more information about the coefficients of solutions to the Fay identity (see Proposition \ref{prop:coefficients}). This in turn sheds some light on how solutions to the associative Yang--Baxter equation can be constructed algorithmically.

\subsection{Content}
Section \ref{sec:2} is preliminary; we recall the definition of the Kronecker function $F_{\tau}(u,v)$ and state the Fay identity. This identity is then studied in detail in Section \ref{sec:3} which culminates in the proof of Theorem \ref{thm:main}. In Section 4 we explain the relation between Fay identity and period relations, building on Zagier's fundamental result, and finish by giving a proof of Theorem \ref{thm:main2}.
\smallskip

{\bf Acknowledgements:} This work was done while the author was a JSPS Postdoctoral Fellow, partly supported by JSPS KAKENHI Grant No. 17F17020. The author would like to thank his academic host, Professor Masanobu Kaneko, for his support and fruitful discussions on the contents of this manuscript, as well as Ulf K\"uhn and the referees for crucial comments and helpful feedback. Final corrections were made while the author was a Postdoctoral Research Assistant at University of Oxford, partly supported by ERC grant 724638.
\smallskip

{\bf Notation:}
Given variables $X_1,\ldots,X_n$ and a ring $R$, we will denote by $R((X_1,\ldots,X_n))$ the $R$-algebra of formal Laurent series of the form
\begin{equation}
\sum_{m_1,\ldots,m_n>\!\!>-\infty}a_{m_1,\ldots,m_n}X_1^{m_1}\ldots X_n^{m_n}, \quad \mbox{with } a_{m_1,\ldots,m_n} \in R.
\end{equation}

\section{The Kronecker function and its functional equations} \label{sec:2}

The reference for this section is \cite{Zagier:Periods}.

\subsection{The Fay identity for the Kronecker function} \label{ssec:2.1}
Consider the classical odd Jacobi theta function $\theta_{\tau}(u):=\sum_{n\in \bZ}(-1)^nq^{\frac 12\left(n+\frac 12\right)^2}e^{\left(n+\frac 12\right)u}$. It is entire, has simple zeros exactly for $u \in 2\pi i(\bZ+\bZ\tau)$ and satisfies $\theta_{\tau}(u+2\pi i(m\tau+r))=(-1)^{m+r}q^{-\frac{m^2}{2}}e^{-mu}\theta_{\tau}(u)$, for $m,r\in \bZ$.
\begin{dfn}
The Kronecker function is the meromorphic function $F_{\tau}: \bC^2 \rightarrow \bC$, defined by
\begin{equation}
F_{\tau}(u,v)=\frac{\theta'_{\tau}(0)\theta_{\tau}(u+v)}{\theta_{\tau}(u)\theta_{\tau}(v)}.
\end{equation}
\end{dfn}
From the properties of $\theta_{\tau}$ mentioned above, it follows that $F_{\tau}(u,v)$ has simple poles if $u$ or $v$ is in $2\pi i(\bZ+\bZ\tau)$, simple zeros for $u+v \in 2\pi i(\bZ+\bZ\tau)$ and satisfies $F_{\tau}(u+2\pi i(n\tau+s),v+2\pi i(m\tau+r))=q^{-mn}e^{-mu-nv}F_{\tau}(u,v)$, for $m,n,r,s\in \bZ$. Also, at the cusp $i\infty$, the Kronecker function degenerates to a trigonometric function:
\begin{equation}
F_{\tau}(u,v)\vert_{\tau=i\infty}=\frac 12\left(\coth\left( \frac u2\right)+\coth\left( \frac v2 \right)\right).
\end{equation}
\begin{rmk}
The above version of the Kronecker function is the one given in \cite{Zagier:Periods}. Other sources such as \cite{BL:MEP,Polishchuk:YB} use the function $2\pi iF_{\tau}(2\pi iu,2\pi iv)$ instead.
\end{rmk}
We are interested in functional equations satisfied by $F_{\tau}$. Since $\theta_{\tau}(-u)=-\theta_{\tau}(u)$, the Kronecker function is antisymmetric, i.e.
$
F_{\tau}(u,v)+F_{\tau}(-u,-v)=0.
$
More interestingly, it satisfies the following three-term functional equation, \cite{BL:MEP}, Proposition 5.(iii).\footnote{Note that our version is slightly different, see Remark \ref{rmk:Fay} for more details.}
\begin{prop} \label{prop:Fay}
The Kronecker function satisfies the Fay identity,
\begin{equation} \label{eqn:Fay}
F_{\tau}(u_1,v_1)F_{\tau}(u_2,v_2)+F_{\tau}(-u_2,v_1-v_2)F_{\tau}(u_1+u_2,v_1)+F_{\tau}(-u_1-u_2,-v_2)F_{\tau}(u_1,v_1-v_2)=0.
\end{equation}
\end{prop}
\begin{proof}
We give a proof for convenience of the reader. Writing out the definition of $F_{\tau}(u,v)$ and multiplying by the common denominator of the left hand side of \eqref{eqn:Fay}, we see that \eqref{eqn:Fay} is equivalent to
\begin{equation}\label{eqn:propFay2}
\begin{aligned}
0&=\theta_{\tau}(u_1+v_1)\theta_{\tau}(u_2+v_2)\theta_{\tau}(u_1+u_2)\theta_{\tau}(v_1-v_2)\\
&-\theta_{\tau}(-u_2+v_1-v_2)\theta_{\tau}(u_1+u_2+v_1)\theta_{\tau}(v_2)\theta_{\tau}(u_1)\\
&+\theta_{\tau}(-u_1-u_2-v_2)\theta_{\tau}(u_1+v_1-v_2)\theta_{\tau}(u_2)\theta_{\tau}(v_1),
\end{aligned}
\end{equation}
where we also used $\theta_{\tau}(-u)=-\theta_{\tau}(u)$. Now substituting
\begin{alignat}{3}
&\alpha_0=u_1+u_2+v_2, \quad &&\alpha_1=-u_2+v_1-v_2, \quad &&\alpha_2=-\alpha_0-\alpha_1=-u_1-v_1,\\
&\beta_0=u_2, \quad &&\beta_1=v_2, &&\beta_2=-\beta_0-\beta_1=-u_2-v_2,
\end{alignat}
and again using antisymmetry of $\theta_{\tau}$, we can write \eqref{eqn:propFay2} in the more symmetric form
\begin{equation}
\sum_{i \in \bZ/3\bZ}\theta_{\tau}(\alpha_i)\theta_{\tau}(\beta_i)\theta_{\tau}(\alpha_{i-1}+\beta_{i+1})\theta_{\tau}(\alpha_{i+1}-\beta_{i-1})=0,
\end{equation}
and this is precisely \cite[Proposition 5]{Zagier:Periods}.
\end{proof}
Slightly more generally, we have the following result.
\begin{cor} \label{cor:solution}
The functions
\begin{equation}
e^{\gamma uv}\left( \frac{\alpha}{u}+\frac{\beta}{v} \right), \quad \mbox{and} \quad \delta e^{\gamma uv}F_{\tau}\left( \frac{u}{\alpha},\frac{v}{\beta} \right),
\end{equation}
both satisfy the Fay identity, for all $\alpha,\beta,\gamma \in \bC$ in the first case and for all  $\alpha,\beta \in \bC^{\times}$, $\gamma,\delta \in \bC$ and $\tau \in \fH \cup \{i\infty\}$ in the second.
\end{cor}
\begin{proof}
The key observation is that the product $e^{\gamma u_1v_1}e^{\gamma u_2v_2}$ is invariant under the following two linear transformations which occur in the Fay identity:
\begin{equation}
\begin{aligned}
(u_1,v_1,u_2,v_2) \mapsto (-u_2,v_1-v_2,u_1+u_2,v_1),\\
(u_1,v_1,u_2,v_2) \mapsto (-u_1-u_2,-v_2,u_1,v_1-v_2).
\end{aligned}
\end{equation}
The corollary then follows from partial fractions in the first case and from Proposition \ref{prop:Fay} in the second.
\end{proof}

\section{Algebraic structure of the Fay identity} \label{sec:3}

In this section, we always let $f(u,v)=\sum_{m,n>\!\!>-\infty}a_{m,n}u^mv^n \in \bC((u,v))$ be a formal Laurent series. The goal of this section is to derive constraints on the coefficients $a_{m,n}$ imposed by the Fay identity. Our main results (Theorem \ref{thm:generators} and Proposition \ref{prop:Faybasic2} below) show that if $f$ satisfies the Fay identity, it is uniquely determined by at most five of its coefficients.
\subsection{Some basic implications of the Fay identity}
We begin by showing how the Fay identity implies the vanishing of many of the coefficients $a_{m,n}$.
\begin{prop} \label{prop:Faybasic}
If $f$ satisfies the Fay identity, then
\begin{equation}
f(u,v)=\frac{a_{-1,0}}{u}+\frac{a_{0,-1}}{v}+\sum_{m,n\geq 0}a_{m,n}u^mv^n, \quad \mbox{with }a_{m,n}=0, \, \mbox{if } m+n \in 2\bZ.
\end{equation}
In particular, $f(-u,-v)=-f(u,v)$. Moreover, if $(a_{-1,0},a_{0,-1})=(0,0)$, then $f=0$.
\end{prop}
\begin{proof}
We first show that if $f \neq 0$, then $f$ must have a pole at either $u=0$ or $v=0$. Indeed, if $f$ does not have a pole at $v=0$, then the Fay identity implies
\begin{align}
f(u_1,0)f(u_2,0)+f(-u_2,0)f(u_1+u_2,0)+f(-u_1-u_2,0)f(u_1,0)&=0,\\
f(u_1,v_1)f(u_2,v_1)+f(-u_2,0)f(u_1+u_2,v_1)+f(-u_1-u_2,-v_1)f(u_1,0)&=0.
\end{align}
The first equation implies $f(u,0)=a_{-1,0}u^{-1}$, as can be seen by comparing coefficients. In particular, if $f$ does not have a pole at $u=0$, we must have $f(u,0)=0$, but then the second equation implies $f(u,v)=0$.

Now assume that $f \neq 0$ and let $M,N \in \bZ$ be the largest integers such that $a_{-M,n} \neq 0$ for some $n \in \bZ$ and $a_{m,-N} \neq 0$ for some $m \in \bZ$. Since $f$ has a pole at either $u=0$ or $v=0$, we have $M\geq 1$ or $N\geq 1$. Define
\begin{equation}
g(v):=(u^Mf(u,v))\vert_{u=0}=\sum_{n>\!\!>-\infty}a_{-M,n}v^n, \quad h(u):=(v^Nf(u,v))\vert_{v=0}=\sum_{m>\!\!>-\infty}a_{m,-N}u^m.
\end{equation}
Both $g$ and $h$ are well-defined by construction of $M$ and $N$. Now, if $f$ has a pole at $u=0$, we multiply the Fay identity by $(u_1u_2)^M$ and then set $u=u_1=u_2=0$ to get
\begin{equation}
g(v_1)g(v_2)+\left(\frac{-1}{2}\right)^Mg(v_1-v_2)g(v_1)+\left(\frac{-1}{2}\right)^Mg(-v_2)g(v_1-v_2)=0,
\end{equation}
and it is straightforward to verify that this implies $M=1$ and that $g(v)=a_{-1,0}$. Likewise, if $f(u,v)$ has a pole at $v=0$ (i.e. $N\geq 1$), then a similar argument yields $N=1$ and $h(u)=a_{0,-1}$. This shows that
\begin{equation}
f(u,v)=\frac{a_{-1,0}}{u}+\frac{a_{0,-1}}{v}+\sum_{m,n \geq 0}a_{m,n}u^mv^n,
\end{equation}
with $(a_{-1,0},a_{0,-1}) \neq (0,0)$ if $f\neq 0$. 

It remains to prove that $a_{m,n}=0$ if $m+n$ is even which is equivalent to antisymmetry $f(-u,-v)=-f(u,v)$. For this we may clearly assume that $f\neq 0$. Taking the residues of the Fay identity at $u_1=0$, respectively at $v_1=0$, gives
\begin{align}
a_{-1,0}f(u_2,v_2)+a_{-1,0}f(-u_2,-v_2)&=0,\\
a_{0,-1}f(u_2,v_2)+a_{0,-1}f(-u_2,-v_2)&=0,
\end{align}
and since $(a_{-1,0},a_{0,-1}) \neq (0,0)$, the result follows.
\end{proof}
\begin{rmk} \label{rmk:Fay}
We have already mentioned that our version of the Fay identity is slightly different from the one in \cite[Proposition 5]{BL:MEP}. In particular, the latter does not imply antisymmetry. Indeed, for every $\alpha \neq 0$ the function $f(u,v):=\alpha (\coth(\alpha u)+1)$ satisfies the Fay identity as given in \cite{BL:MEP} (with $u$ corresponding to $\xi$), but does not satisfy Equation \eqref{eqn:Fay}. On the other hand, antisymmetry together with the version of the Fay identity given in \textit{loc.cit.} are equivalent to \eqref{eqn:Fay}. Therefore our version of the Fay identity subsumes antisymmetry as well which is the reason why we prefer to work with it.
\end{rmk}
By Proposition \ref{prop:Faybasic}, we already know that $f=0$, if $(a_{-1,0},a_{0,-1})=(0,0)$. Therefore, the next proposition finishes the proof of Theorem \ref{thm:main} in the special case $a_{-1,0}a_{0,-1}=0$.
\begin{prop} \label{prop:Faybasic2}
Assume that $f \neq 0$ satisfies the Fay identity.
\begin{enumerate}
\item[(i)]If $a_{0,-1}=0$, then $f$ is uniquely determined by $a_{-1,0}$ and $a_{0,1}$, and we have
\begin{equation}
f(u,v)=\frac{a_{-1,0} e^{\gamma uv}}{u}, \quad \mbox{with }\gamma=\frac{a_{0,1}}{a_{-1,0}}.
\end{equation}
\item[(ii)] 
If $a_{-1,0}=0$, then $f$ is uniquely determined by $a_{0,-1}$ and $a_{1,0}$, and we have
\begin{equation}
f(u,v)=\frac{a_{0,-1} e^{\gamma uv}}{v}, \quad \mbox{with }\gamma=\frac{a_{1,0}}{a_{0,-1}}.
\end{equation}
\end{enumerate}
\end{prop}
Note that $\gamma$ is well-defined in both cases. Indeed, since $f\neq 0$ we must have $a_{-1,0} \neq 0$ in the first case and $a_{0,-1} \neq 0$ in the second one, by Proposition \ref{prop:Faybasic}.
\begin{proof}
We only prove (i), the proof of (ii) is analogous. Since $a_{0,-1}=0$, Proposition \ref{prop:Faybasic} implies that $f(u,v)$ does not have a pole at $v=0$ and the Fay identity with $v:=v_1=v_2$ and $u:=u_1=u_2$ yields
\begin{equation} \label{eqn:prop2}
f(u,v)^2-2f(u,0)f(2u,v)=0.
\end{equation}
Writing out the Laurent expansion of $f$ and using that $f(u,0)=a_{-1,0}u^{-1}$ (see the proof of Proposition \ref{prop:Faybasic}), Equation \eqref{eqn:prop2} is equivalent to
\begin{equation} \label{eqn:prop3}
\left(\sum_{m,n\geq 0}a_{m,n}u^mv^n\right)^2-a_{-1,0}\sum_{m\geq 1, \, n\geq 0}(2^{m+1}-2)a_{m,n}u^{m-1}v^n=0.
\end{equation}
This implies $a_{m,0}=0$ for $m\neq -1$, and then by induction on $n$ we get $a_{m,n}=0$, if $m \neq n-1$. More generally,
\begin{equation}
\left(\sum_{\genfrac{}{}{0pt}{}{1\leq i,j\leq n-1}{i+j=n}}a_{i-1,i}a_{j-1,j}\right)u^{n-2}v^n-a_{-1,0}\sum_{m\geq 1}(2^{m+1}-2)a_{m,n}u^{m-1}v^n=0,
\end{equation}
for every $n$, showing that $a_{n-1,n}$ is recursively determined by $a_{-1,0}$ and $a_{0,1}$, and therefore $f$ itself is uniquely determined by $a_{-1,0}$ and $a_{0,1}$. On the other hand, by Corollary \ref{cor:solution} there exists a solution to the Fay identity for any given values of $a_{-1,0} \in \bC^{\times}$ and $a_{0,1} \in \bC$ namely 
$
\alpha e^{\gamma uv}u^{-1}
$
with $\alpha=a_{-1,0}$ and $\gamma=a_{0,1}/a_{-1,0}$, and this ends the proof.
\end{proof}

\subsection{The ideal of Fay relations}
To study the Fay identity in more detail, it will be convenient to replace the coefficients $a_{m,n} \in \bC$ by symbols $A_{m,n}$ and accordingly to study a formal version of the Fay identity. More precisely, consider the polynomial $\bC$-algebra
\begin{equation} \label{eqn:defA}
\cA:=\bC[\{A_{-1,0},A_{0,-1}\} \cup \{A_{m,n} \, \vert \, m,n \geq 0, \, m+n \mbox{ odd}\}]
\end{equation}
(the restriction on the indices $(m,n)$ is justified by Proposition \ref{prop:Faybasic}) and let 
\begin{equation}
\Phi(u,v)=\sum A_{m,n}u^mv^n \in (uv)^{-1}\cA[[u,v]]
\end{equation}
be the generic element where the sum is over all $(m,n)$ as in \eqref{eqn:defA}. By definition it satisfies $\Phi(-u,-v)=-\Phi(u,v)$. Also, define
\begin{equation} \label{eqn:Faygeneric}
\begin{aligned}
\cF(u_1,u_2,v_1,v_2):=\Phi(u_1,v_1)\Phi(u_2,v_2)&+\Phi(-u_2,v_1-v_2)\Phi(u_1+u_2,v_1)\\
&+\Phi(-u_1-u_2,-v_2)\Phi(u_1,v_1-v_2).
\end{aligned}
\end{equation}
A priori, it is contained in $(u_1u_2v_1v_2(u_1+u_2)(v_1-v_2))^{-1}\cA[[u_1,u_2,v_1,v_2]]$.
\begin{lem}
The element $\cF(u_1,u_2,v_1,v_2)$ is contained in $\cA[[u_1,u_2,v_1,v_2]]$, i.e.
\begin{equation}
\cF(u_1,u_2,v_1,v_2)=\sum_{m_1,m_2,n_1,n_2 \geq 0}c_{m_1,m_2,n_1,n_2}u_1^{m_1}u_2^{m_2}v_1^{n_1}v_2^{n_2},
\end{equation}
for some $c_{m_1,m_2,n_1,n_2} \in \cA$.
\end{lem}
\begin{prf}
By definition, $\Phi(u,v)$ has simple poles exactly along $u=0$ and $v=0$ with residues $A_{-1,0}$ and $A_{0,-1}$ respectively. It is therefore sufficient to check that all residues of $\cF$ along $u_1=0$, $u_2=0$, $u_1+u_2=0$, $v_1=0$, $v_2=0$ and $v_1-v_2=0$ vanish, which is straightforward and essentially only uses that $\Phi(-u,-v)=-\Phi(u,v)$.
\end{prf}

\begin{dfn}
Define $\cJ \subset \cA$ to be the ideal generated by the coefficients $c_{m_1,m_2,n_1,n_2}$ of $\cF$. Also, define $\bar\cA:=\cA/\cJ$ to be the corresponding quotient.
\end{dfn}
The ideal $\cJ$ could be called the ideal of Fay relations. By definition, giving a formal Laurent series $f(u,v)=\sum_{m,n>\!\!>-\infty}a_{m,n}u^mv^n \in \bC((u,v))$ which satisfies the Fay identity is equivalent to giving a homomorphism of $\bC$-algebras 
\begin{equation}
\varphi_f: \bar\cA \rightarrow \bC, \quad \bar A_{m,n} \mapsto a_{m,n}.
\end{equation}
Solutions to the Fay identity with $a_{-1,0}a_{0,-1} \neq 0$ correspond under this identification to homomorphisms 
$
\varphi_f: \bar \cA_0 \rightarrow \bC,
$
where $\bar \cA_0:=\bar \cA \otimes_{\bC} \bC[\bar A^{-1}_{-1,0},\bar A^{-1}_{0,-1}]$. By Proposition \ref{prop:Faybasic2} it is enough to classify the latter, and we shall therefore be interested in understanding the structure of $\bar\cA_0$. To this end, let 
$
\bar\cA' \subset \bar\cA$,
be the $\bC$-subalgebra generated by the set $\{\bar A_{0,-1}\} \cup \{\bar A_{m,0} \, \vert \, m \in \{(-1,1,3,5)\}\}$ and denote by
$
\iota: \bar\cA' \hookrightarrow \bar\cA$
the canonical inclusion. Extending scalars to $\bC[\bar A^{-1}_{-1,0},\bar A^{-1}_{0,-1}]$, we get an induced map
\begin{equation}
\iota_0: \bar\cA'_0\rightarrow \bar\cA_0, \quad \mbox{where } \bar\cA'_0:=\bar \cA' \otimes_{\bC} \bC[\bar A^{-1}_{-1,0},\bar A^{-1}_{0,-1}],
\end{equation}
which is clearly injective.
\begin{thm} \label{thm:generators}
The map $\iota_0$ is an isomorphism of algebras.
\end{thm}
More concretely, using the $1$-$1$ correspondence $f \leftrightarrow \varphi_f$ described above, Theorem \ref{thm:generators} says that every solution $f(u,v) \in \bC((u,v))$ to the Fay identity with $a_{-1,0}a_{0,-1} \neq 0$ is uniquely determined by its coefficients $a_{0,-1}$ and $a_{m,0}$, for $m=-1,1,3,5$.
\begin{rmk}
Although we will not need this, one can show that $\bar \cA_0$ is freely generated, as a $\bC[\bar A^{-1}_{-1,0},\bar A^{-1}_{0,-1}]$-algebra, by $\bar A_{0,-1}$ and $\bar A_{m,0}$ for $m=-1,1,3,5$; this follows from Theorem \ref{thm:generators} since for every quintuple $(z_{-1},w_{-1},z_1,z_3,z_5) \in \bC^5$ such that $z_{-1}w_{-1} \neq 0$ there exists a solution $f(u,v)=\sum_{m,n >\!\!> -\infty}a_{m,n}u^mv^n$ to the Fay identity such that $a_{0,-1}=w_{-1}$ and $a_{m,0}=z_m$ for $m=-1,1,3,5$ (see the proof of Theorem \ref{thm:main}).
\end{rmk}

Before we prove Theorem \ref{thm:generators}, we need to introduce some more notation. Let $\fp \subset \cA$ be the ideal generated by all $A_{m,n}$ with $m,n\geq 0$. This is a homogeneous prime ideal of $\cA$, the grading being defined by giving $A_{m,n}$ degree $m+n$. Moreover, since the ideal $\cJ$ of Fay relations is homogeneous, this grading descends to the quotient $\bar \cA$. In general, given a homogeneous ideal $\cI$ of either $\cA$ or $\bar \cA$, we will denote by $\cI_k$ its component of degree $k$.

The following proposition gives explicit formulas for some of the coefficients $c_{m_1,m_2,n_1,n_2}$ and will be the key for proving Theorem \ref{thm:generators}.
\begin{prop} \label{prop:coefficients} 
We have the following formulas for the coefficients of $\cF$:
\begin{equation}
c_{0,0,0,0}=3A_{0,-1}A_{0,1}-3A_{-1,0}A_{1,0} \label{eqn:1}
\end{equation}
and for $k\geq 2$ even,
\begin{align}
c_{0,0,0,k}& \equiv (k+3)A_{0,-1}A_{0,k+1}-2A_{-1,0}A_{1,k} \mod \fp^2_k, \label{eqn:2}\\
c_{0,k,0,0}&\equiv 2A_{0,-1}A_{k,1}-(k+3)A_{-1,0}A_{k+1,0} \mod \fp^2_k, \label{eqn:3}\\
c_{0,0,2,k-2}&\equiv \left( \binom{k+2}{3}+1 \right)A_{0,-1}A_{0,k+1}-\left(\binom{k}{2}+\delta_{k,2}\right)A_{-1,0}A_{1,k} \mod \fp^2_k, \label{eqn:4}\\
c_{k-2,2,0,0}&\equiv -\left( \binom{k+2}{3}+1 \right)A_{-1,0}A_{k+1,0}+\left(\binom{k}{2}+\delta_{k,2}\right)A_{0,-1}A_{k,1} \mod \fp^2_k. \label{eqn:5}
\end{align}
Finally, for $0<m<k$ with $k$ as above, we have
\begin{equation}
c_{0,m,0,k-m}\equiv(k+2-m)A_{0,-1}A_{m,k+1-m}-(m+2)A_{-1,0}A_{m+1,k-m} \mod \fp^2_k. \label{eqn:6}
\end{equation}
\end{prop}
\begin{proof}
Since the Fay identity is homogeneous, for every monomial $u_2^rv_2^s$ with $r,s\geq 0$ the coefficient of $u_2^rv_2^s$ in \eqref{eqn:Faygeneric} is congruent modulo $\fp^2_{r+s}$ to the coefficient of $u_2^rv_2^s$ in
\begin{align} \label{eqn:Fayspecial}
\varphi_{r,s}(u_1,v_1)\varphi_{r,s}(u_2,v_2)&+\varphi_{r,s}(-u_2,v_1-v_2)\varphi_{r,s}(u_1+u_2,v_1)\\
&+\varphi_{r,s}(-u_1-u_2,-v_2)\varphi_{r,s}(u_1,v_1-v_2),
\end{align}
where
$
\varphi_{r,s}(u,v)=A_{-1,0}u^{-1}+A_{0,-1}v^{-1}+A_{r,s+1}u^rv^{s+1}+A_{r+1,s}u^{r+1}v^{s}.
$
A straightforward computation of the coefficient of $u_2^mv_2^{k-m}$ in \eqref{eqn:Fayspecial} now yields \eqref{eqn:1}, \eqref{eqn:2}, \eqref{eqn:3} and $\eqref{eqn:6}$.

Similarly, for $k\geq 2$ the coefficient of $v_1^2v_2^{k-2}$ (respectively of $u_1^{k-2}u_2^2$) in \eqref{eqn:Faygeneric} is congruent modulo $\fp^2_k$ to the coefficient of the same monomial in \eqref{eqn:Fayspecial} for $(r,s)=(0,k)$ (respectively for $(r,s)=(k,0)$), and we get \eqref{eqn:4} and \eqref{eqn:5}.
\end{proof}
\begin{proof}[Proof of Theorem \ref{thm:generators}]
It is clearly enough to show that $\bar A_{m,n} \in \bar \cA'_0$ for all $m,n$. We prove this by induction on the degree $d=m+n$ of $\bar A_{m,n}$. For $d=1$, we have $\bar A_{1,0} \in \bar \cA'$ by definition and it follows from \eqref{eqn:1} that $\bar A_{0,1}=\bar A_{-1,0}^{-1}\bar A_{0,-1}\bar A_{1,0} \in \bar \cA'_0$. Note that this also implies that $(\bar\fp^2_0)_2 \subset \bar\cA'_0$ where $\bar\fp^2_0$ denotes the ideal generated by the image of $\fp^2$ in $\bar\cA_0$.

Now we use induction on $d$ to show that $\bar A_{m,n} \in \bar\cA'_0$ for all $m,n$ with $d=m+n$ (this shows in particular that $(\bar\fp^2_0)_{d-1} \subset \bar\cA'_0$). For $d=3$ or $d=5$, since $\bar A_{d,0} \in \bar\cA'$ in that case, we see from \eqref{eqn:3} together with $(\bar\fp^2_0)_{d-1} \subset \bar\cA'_0$ (which follows from the induction hypothesis) that $\bar A_{d-1,1} \in \bar\cA'_0$. Repeating the same argument, using \eqref{eqn:6} and finally \eqref{eqn:2}, we obtain $\bar A_{m,n} \in \bar\cA'_0$ for all $m+n=d$, if $d=3$ or $d=5$.

If $d\geq 7$, the crucial point is that the row vectors $(-(k+3),2)$ and $(-\binom{k+2}{3}-1,\binom{k}{2}+\delta_{k,2})$, where $k=d-1$, are linearly independent. Therefore, from \eqref{eqn:3} and \eqref{eqn:5}, we see that $\bar A_{d,0} \equiv 0 \mod (\bar\fp^2_0)_{d-1}$, i.e. $\bar A_{d,0} \in (\bar\fp^2_0)_{d-1}$. By our induction hypothesis, this shows $\bar A_{d,0} \in \bar\cA'_0$ and using a similar argument as before we get $\bar A_{m,n} \in \bar\cA'_0$ for all $m+n=d$.
\end{proof}

\subsection{Proof of Theorem {\ref{thm:main}}} \label{ssec:3.3}
Let $f(u,v)=\sum_{m,n>\!\!>-\infty}a_{m,n}u^mv^n$ be a solution to the Fay identity. By Proposition \ref{prop:Faybasic}, we have
\begin{equation}
f(u,v)=\frac{a_{-1,0}}{u}+\frac{a_{0,-1}}{v}+\sum_{m,n\geq 0}a_{m,n}u^mv^n, \quad \mbox{with }a_{m,n}=0, \, \mbox{if $m+n \in 2\bZ$.}
\end{equation}
The case $a_{-1,0}a_{0,-1}=0$ has already been taken care of by Propositions \ref{prop:Faybasic} and \ref{prop:Faybasic2} so that we may assume $a_{-1,0}a_{0,-1} \neq 0$, i.e. $a_{-1,0}$ and $a_{0,-1}$ are both invertible. We distinguish between two cases.

\noindent
\textit{Case (i): $(a_{3,0},a_{5,0})=(0,0)$.}
Let
\begin{equation}
\frac{b_{-1,0}}{u}+\frac{b_{0,-1}}{v}+\sum_{m,n\geq 0}b_{m,n}u^mv^n,
\end{equation}
be the Laurent expansion of $e^{\gamma uv}( \alpha/u+\beta/v)$.\footnote{Here and in the following we suppress the dependence of the Laurent coefficients $b_{m,n}$ on the parameters $\alpha,\beta,\gamma,\delta,\tau$.} It is straightforward to verify that setting $\alpha=a_{-1,0}$, $\beta=a_{0,-1}$ and $\gamma=a_{1,0}/\beta$, we have $a_{0,-1}=b_{0,-1}$ and $a_{m,0}=b_{m,0}$ for $m=-1,1,3,5$. Therefore $a_{m,n}=b_{m,n}$ for all $m,n$ by Theorem \ref{thm:generators}. The uniqueness of $\alpha,\beta$ and $\gamma$ is clear.
\bigskip

\noindent
\textit{Case (ii): $(a_{3,0},a_{5,0}) \neq (0,0)$.}
For any choice of parameters $\alpha,\beta,\gamma,\delta,\tau$ as in the statement of Theorem \ref{thm:main}.(ii), we get the following Laurent expansion
\begin{equation}
\delta e^{\gamma uv}F_{\tau}\left(\frac{u}{\alpha},\frac{v}{\beta}\right)=\frac{b_{-1,0}}{u}+\frac{b_{0,-1}}{v}+\sum_{m,n\geq 0}b_{m,n}u^mv^n.
\end{equation}
It follows from \cite[Theorem 3.(iv)]{Zagier:Periods} that $
b_{-1,0}=\alpha\delta$, $b_{0,-1}=\beta\delta,
$ and
\begin{align}
b_{m,0}=\begin{cases}\displaystyle-2\delta\frac{\alpha^{-m}}{m!}G_{m+1}(\tau)=-2b_{-1,0}\frac{\alpha^{-(m+1)}}{m!}G_{m+1}(\tau), & m \neq 1,\\\\
\displaystyle\beta\gamma\delta-2\delta\alpha^{-1}G_2(\tau)=b_{0,-1}\gamma-2b_{-1,0}\alpha^{-2}G_2(\tau), & m=1,
\end{cases}
\end{align}
where $G_k(\tau)$ is the Hecke-normalized Eisenstein series of weight $k$ (the normalization is in fact irrelevant here, we will only need that $G_k(\tau)$ is a modular form of weight $k$).

We now view $b_{3,0}$ and $b_{5,0}$ as functions of $(\alpha,\tau)$ and claim that the map
\begin{equation}
\bC^{\times} \times \overline{\fH} \rightarrow \bC^2 \setminus \{(0,0)\}, \quad (\alpha,\tau) \mapsto (b_{3,0},b_{5,0}),
\end{equation}
where $\overline{\fH}:=\fH \cup \{i\infty\}$, is surjective. Indeed, it is enough to prove that the map $(\alpha,\tau) \mapsto (b_{3,0}^3,b_{5,0}^2)$ is surjective which in turn is equivalent to proving surjectivity of
\begin{equation}
\overline{\fH} \rightarrow \bP^1(\bC), \quad \tau \mapsto [b_{3,0}^3:b_{5,0}^2].
\end{equation}
But this map is surjective because the quotient $b_{3,0}^3/b_{5,0}^2$ is a non-constant modular function (being proportional to $G_4(\tau)^3/G_6(\tau)^2$), and every such function surjects onto $\bP^1(\bC)$.

By what we have just proved, given any coefficients $(a_{3,0},a_{5,0}) \in \bC^2 \setminus \{(0,0)\}$, we can choose $\alpha \in \bC^{\times}$ and $\tau \in \overline{\fH}$ such that we have an equality of Laurent coefficients $b_{3,0}=a_{3,0}$ and $b_{5,0}=a_{5,0}$. Now setting
\begin{equation}
\delta=\frac{a_{-1,0}}{\alpha}, \quad \beta=\frac{a_{0,-1}}{\delta}, \quad \gamma=\frac{a_{1,0}+2a_{-1,0}\alpha^{-2}G_2(\tau)}{a_{0,-1}},
\end{equation}
we get $a_{0,-1}=b_{0,-1}$ as well as $a_{m,0}=b_{m,0}$ for $m=-1,1,3,5$, hence $a_{m,n}=b_{m,n}$ for all $m,n$ by Theorem \ref{thm:generators}.

On the other hand, assume that $\alpha',\beta',\gamma',\delta',\tau'$ are different parameters such that $f(u,v)=\delta'e^{\gamma'uv}F_{\tau'}\left( \frac{u}{\alpha'},\frac{v}{\beta'} \right)$. Then, since the quotient $b_{3,0}^3/b_{5,0}^2$ (which depends neither on $\alpha$, nor on $\alpha'$) is a non-constant modular function, there exists $\left( \begin{smallmatrix}a&b\\c&d\end{smallmatrix} \right) \in \SL_2(\bZ)$ such that $\tau'=\frac{a\tau+b}{c\tau+d}$. The fact that $\alpha'=\alpha(c\tau+d)$ and $\gamma'=\gamma-\frac{c/2\pi i}{\alpha\beta(c\tau+d)}$ then follows by looking at the coefficients $b_{m,0}$ and using the modular transformation property of $G_{m+1}(\tau)$, if $\tau,\tau' \neq i\infty$. Finally, that $\delta'=\frac{\delta}{c\tau+d}$ follows from the previous result using that $\alpha\delta=b_{-1,0}=\alpha'\delta'$ and similarly one shows $\beta'=\beta(c\tau+d)$. This ends the proof of Theorem \ref{thm:main}.
\qed

\section{Proof of the main result} \label{sec:4}

We now give the proof of Theorem \ref{thm:main2}. First, we need to describe the relation between Fay identity and period relations.

\subsection{Comparison with the period relations} \label{ssec:4.1}

We again follow \cite{Zagier:Periods}. Consider the formal generating series
\begin{equation}
C(X,Y;\tau;T)=\frac{(XY-1)(X+Y)}{X^2Y^2}T^{-2}+\sum_{k=2}^{\infty}c_k(X,Y;\tau)\frac{T^{k-2}}{(k-2)!},
\end{equation}
whose coefficients $c_k(X,Y;\tau)$ are given by
\begin{equation}
c_k(X,Y;\tau)=\sum_f \frac{\frac 12(r_f(X)r_f(Y)-r_f(-X)r_f(-Y))}{(2i)^{k-3}\langle f,f\rangle}f(\tau).
\end{equation}
Here the sum is over all normalized Hecke eigenforms $f$ of weight $k$, $r_f(X)$ denotes the extended period polynomial of $f$ and $\langle\cdot,\cdot\rangle$ is the Petersson inner product. It follows from the definition of $C$ and Equation \eqref{eqn:periodrels} that $C$ satisfies the following variant of the period relations
\begin{equation} \label{eqn:periodrelations}
\begin{aligned}
C(X,Y;\tau;T)+C\left(-\frac 1X,Y;\tau;XT\right)&=0,\\
C(X,Y;\tau;T)+C\left(1-\frac 1X,Y;\tau;XT\right)+C\left( \frac{1}{1-X},Y;\tau;(1-X)T \right)&=0.
\end{aligned}
\end{equation}
Using the ``slash operator''
\begin{equation} \label{eqn:action}
(C \vert \gamma)(X,Y;\tau;T):=C\left( \frac{aX+b}{cX+d},Y;\tau;(cX+d)T \right), \quad \gamma=\begin{pmatrix}a&b\\c&d\end{pmatrix} \in \SL_2(\bZ),
\end{equation}
Equation \eqref{eqn:periodrelations} can be expressed more concisely as
$C\vert(1+S)=C\vert(1+U+U^2)=0$, where $S=\left(\begin{smallmatrix}0&-1\\1&0\end{smallmatrix}\right)$ and $U=\left(\begin{smallmatrix}1&-1\\1&0\end{smallmatrix}\right)$ are generators of $\SL_2(\bZ)$ and $\vert$ is extended linearly to the group ring of $\SL_2(\bZ)$.

Now recall from Theorem \ref{thm:Zagier} that $C(X,Y;\tau;T)=F_{\tau}(T,-XYT)F_{\tau}(XT,YT)$. In our context, this gives a relation between the Fay identity and the period relations, which we describe next. Given a formal Laurent series $f(u,v) \in \bC((u,v))$, we define
\begin{equation}
C_f(X,Y,T)=f(T,-XYT)f(XT,YT) \in \bC((X,Y,T)).
\end{equation}
\begin{prop} \label{prop:periodsFay}
The series $C_f$ satisfies the period relations, Equation \eqref{eqn:periodrelations}, if and only if $f$ satisfies the Fay identity, Equation \eqref{eqn:Fay}.
\end{prop}
\begin{proof}
By Proposition \ref{prop:Faybasic}, every solution to the Fay identity is antisymmetric, so we may as well prove that the period relations are equivalent to Fay identity and antisymmetry. The equivalence of antisymmetry for $f$ and the period relation $C_f\vert(1+S)=0$ is straightforward. On the other hand, setting $u_1:=T$, $u_2:=-XT$, $v_1:=-XYT$ and $v_2:=-YT$, the period relation $C_f\vert(1+U+U^2)=0$ is
\begin{equation} \label{eqn:Fayidentity2}
\begin{aligned}
f(u_1,v_1)f(-u_2,-v_2)&+f(-u_2,v_1-v_2)f(-u_1-u_2,-v_1)\\
&+f(-u_1-u_2,-v_2)f(-u_1,v_2-v_1)=0,
\end{aligned}
\end{equation}
which together with antisymmetry implies \eqref{eqn:Fay}. Conversely, Fay identity and antisymmetry together imply \eqref{eqn:Fayidentity2}.
\end{proof}

\subsection{Proof of Theorem {\ref{thm:main2}}} \label{ssec:4.2}

Combining Proposition \ref{prop:periodsFay} and Theorem \ref{thm:main}, we know that $C_f(X,Y,T)$ satisfies the period relations if and only if
\begin{equation}
f(u,v)=e^{\gamma uv}\left(\frac{\alpha}{u}+\frac{\beta}{v}\right), \quad \mbox{ or } \quad f(u,v)=\widetilde{\delta} e^{\gamma uv}F_{\tau}\left( \frac{u}{\alpha},\frac{v}{\beta} \right),
\end{equation}
for some $\alpha,\beta,\gamma \in \bC$ in the first case and for some $\alpha,\beta,\widetilde{\delta} \in \bC^{\times}$, $\gamma \in \bC$, $\tau \in \fH \cup \{i\infty\}$ in the second. With the notation of Theorem \ref{thm:main2}, in the first case we get $C_f(X,Y,T)=C_{P_{\alpha,\beta}}(X,Y,T)$ and in the second case $C_f(X,Y,T)=\delta C(X,Y';\tau;T')$ with $\delta:=\widetilde{\delta}^2$, as claimed. The uniqueness assertion follows from the corresponding statement in Theorem \ref{thm:main}.
\qed

\end{document}